\documentclass[letterpaper, 10 pt, conference]{ieeeconf}

\IEEEoverridecommandlockouts                             
\usepackage{amsmath,color}
\usepackage{amssymb}  
\usepackage{amsfonts}
\usepackage{mathtools}
\usepackage{graphicx}
\usepackage{dsfont}
\usepackage{psfrag}
\usepackage{subfigure,epstopdf}
\usepackage{cite}
\usepackage{bbm}

\usepackage{algorithmicx}
\usepackage{algorithm}
\usepackage{algpseudocode}
\usepackage{float}
\usepackage{hyperref}
\usepackage{cleveref}
\usepackage{wrapfig}

\usepackage{bm}

\usepackage{slashed}
\usepackage{tikz,gincltex}
\usepackage{arydshln}
\usepackage{nccmath}
\usepackage{caption}
\usepackage{MnSymbol}

\usepackage[font=small,labelfont=bf]{caption}

\usepackage{pgfplots,filecontents}
\pgfplotsset{compat=newest}
\usetikzlibrary{positioning,chains,fit,shapes,calc}
\usetikzlibrary{backgrounds}
\usetikzlibrary{shapes,snakes}
\usetikzlibrary{shapes.geometric, arrows.meta}

\tikzstyle{arrow} = [thick,->,>=stealth]

\usepgfplotslibrary{statistics}

\newtheorem{assumption}{Assumption}
\newtheorem{theorem}{Theorem}
\newtheorem{corollary}{Corollary}

\newtheorem{prop}{Proposition}
\newtheorem{lemma}{Lemma}

\makeatletter
\renewcommand*{\@begintheorem}[2]{\trivlist
      \item[\hskip \labelsep{\bfseries #1\ #2}]\itshape}
\renewcommand*{\@opargbegintheorem}[3]{\trivlist
      \item[\hskip \labelsep{\bfseries #1\ #2}] {(#3)}\ \itshape}
\makeatother

\definecolor{c1}{rgb}{0.368417, 0.506779,0.709798}
\definecolor{c2}{rgb}{0.880722, 0.611041,0.142051}
\definecolor{c3}{rgb}{0.560181, 0.691569,0.194885}
\definecolor{c4}{rgb}{0.922526, 0.385626,0.209179}
\definecolor{c5}{rgb}{0.528488, 0.470624,0.701351}
\definecolor{c6}{rgb}{0.772079, 0.431554,0.102387}
\definecolor{c7}{rgb}{0.363898, 0.618501,0.782349}
\definecolor{ao(english)}{rgb}{0.0, 0.5, 0.0}
\definecolor{cadmiumred}{rgb}{0.89, 0.0, 0.13}

\definecolor{blue}{rgb}{0, 0,0}

\newcommand{\MWDAP}{\textsf{\emph{MWDAP}}}
\newcommand{\MCDAP}{\textsf{\emph{MCDAP}}}
\newcommand{\MWMCDAP}{\textsf{\emph{MWMCDAP}}}

\newcommand{\outneigh}[1]{\mathcal N^{\text{out}}_{#1}}
\newcommand{\inneigh}[1]{\mathcal N^{\text{in}}_{#1}}
\newcommand{\representative}[1]{#1^\star}
\newcommand{\send}[3]{#1 \gg_{\text{{\tiny $#3$}}} #2}
\newcommand{\sendmsg}[4]{#1 \gg_{\text{{\tiny $#3$}}} #2\,:\,#4}
\newcommand{\receive}[3]{#1 \ll_{\text{{\tiny $#3$}}} #2}
\newcommand{\receivemsg}[4]{#1 \ll_{\text{{\tiny $#3$}}} #2\,:\,#4}

\newcommand{\sendlong}[3]{#1 \ggg_{\text{{\tiny $#3$}}} #2}

\newcommand{\sendlongmsg}[4]{#1 \ggg_{\text{{\tiny $#3$}}} #2\,:\,#4}

\newcommand{\receivelongmsg}[4]{#1 \lll_{\text{{\tiny $#3$}}} #2\,:\,#4}

\let\comment\relax
\newcommand{\comment}[1]{{\tiny{\color{gray}  			 \hfill$\diamond$\, {#1}}}}
\newcommand{\phase}[1]{{\scriptsize{\color{gray}  			 \hfill *\quad\textsc{#1}\quad* \hfill}}}
\newcommand{\commentContinue}[1]{{\tiny{\color{gray}  			 \hfill\, {#1}}}}

\tikzstyle{startstop} = [rectangle, rounded corners, minimum width=3cm, minimum height=1cm, text centered, draw=black, fill=c1]
\tikzstyle{process} = [rectangle, minimum width=3cm, minimum height=1cm, text centered, draw=c2, fill=c2!40]
\tikzstyle{io} = [rectangle, rounded corners, minimum width=3cm, minimum height=1cm, text centered, draw=c3, fill=c3!40]
\tikzstyle{io2} = [rectangle, rounded corners, minimum width=3cm, minimum height=1cm, text centered, draw=c3, fill=c3!40]
\tikzstyle{io3} = [diamond, rounded corners, minimum width=3cm, minimum height=1cm, text centered, draw=c4, fill=c4!40]
\tikzstyle{io4} = [rectangle, minimum width=3cm, minimum height=1cm, text centered, draw=c1, fill=c1!40]

\def\firstcircle{(90:.8cm) circle (1.1cm)}
\def\secondcircle{(210:.8cm) circle (1.1cm)}
\def\thirdcircle{(330:.8cm) circle (1.1cm)}

\newfloat{alg}{htbp}{loa}
\floatname{alg}{Algorithm}

\title{\LARGE \bf  Distributed Graph Augmentation Protocols to Achieve Strong Connectivity in Multi-Agent Networks} 

\author{Guilherme Ramos$^\ast$, Diogo Poças$^\ast$, and Sérgio Pequito
    \thanks{$^\ast$The first two authors contributed equally to this work.} 
	\thanks{G. Ramos ({\tt\small guilherme.ramos@tecnico.ulisboa.pt}) is with Dept. of Computer Science and Engineering, Instituto Superior Técnico, University of Lisbon, Portugal and Instituto de Telecomunicações, 1049-001 Lisbon, Portugal.}
	\thanks{D. Poças ({\tt\small diogo.pocas@tecnico.ulisboa.pt}) is with Dept. of Mathematics, Instituto Superior Técnico, University of Lisbon, Portugal and Instituto de Telecomunicações, 1049-001 Lisbon, Portugal. }
	\thanks{S. Pequito ({\tt\small sergio.pequito@tecnico.ulisboa.pt}) is with Dept. of Electrical and Computer Engineering, Instituto Superior Técnico, University of Lisbon, Portugal and Institute for Systems and Robotics (ISR/IST), LARSyS, Lisbon, Portugal.}
 }

\date{} 

\begin{document}
\maketitle

\thispagestyle{empty}
\pagestyle{empty}

\begin{abstract}                          

In multi-agent systems, strong connectivity of the communication network is often crucial for establishing consensus protocols, which underpin numerous applications in decision-making and distributed optimization. 
However, this connectivity requirement may not be inherently satisfied in geographically distributed settings. 
Consequently, we need to find the minimum number of communication links to add to make the communication network strongly connected. 
To date, such problems have been solvable only through centralized methods. 
{\color{blue}This paper introduces a fully distributed algorithm that efficiently identifies an optimal set of edge additions to achieve strong connectivity, using only local information. 
The majority of the communication between agents is local (according to the digraph structure), with only a few steps requiring communication among non-neighboring agents to establish the necessary additional communication links.}   
A comprehensive empirical analysis of the algorithm's performance on various random communication networks reveals its efficiency and scalability. 
\end{abstract}


\section{Introduction}\label{sec:intro}

Network strong connectivity plays a pivotal role in \mbox{multi-agent} networks, encompassing critical applications in communication and coordination among nodes, particularly in applications such as distributed consensus~\cite{garin2010survey,ramossilvestresIJoC, ramos2021distributed, amirkhani2022consensus, ramos2023designing} and distributed optimization~\cite{rabbat2004distributed,huang2015differentially,molzahn2017survey,nedic2018distributed,yang2019survey}. 
In distributed computing and systems where nodes interact and share resources, strong connectivity ensures that tasks can be distributed effectively among nodes~\cite{chlebus2002deterministic,li2006construction,bodlaender2013connectivity}.    

However, in real-world multi-agent settings, this connectivity requirement may not be inherently satisfied. Consequently, solving the digraph augmentation problem -- finding the minimum number of communication links (edges) to add to make the network strongly connected -- becomes essential. To date, such problems have been solvable only through centralized methods. The digraph augmentation problem with varying costs for adding edges is, in general, intractable, i.e., NP-complete~\cite{eswaran1976augmentation,doi:10.1137/0210019,crescenzi1995compendium}, even in the case of a symmetric weighted digraph~\cite{chen1989strongly}. Nevertheless, the works in~\cite{eswaran1976augmentation} and~\cite{rosenthal1977smallest} efficiently solved the digraph augmentation problem when the cost of adding an edge is uniform, i.e., when the objective is to attain an augmentation with the minimal number of edges. 

Notably, existing solutions like~\cite{doi:10.1137/0205044} rely on centralized operations, making them ill-suited for distributed implementation and highlighting the need for novel approaches to operate effectively within the constraints of local information exchange in multi-agent networks. 
It is also worth noticing that parallel computing methods, which utilize multiple processors within a single machine, have been applied to the digraph augmentation problems~\cite{chaudhuri19870,chaudhuri1988fast,aggarwal1989parallel,itokawa2007parallel}. However, these approaches do not readily translate to distributed settings, where tasks are divided across different agents in a multi-agent network, and where each agent gathers only local information, i.e., from its neighbors. 
Hence, distributed approaches pose unique challenges for graph augmentation problems. 

Graph augmentation problems can be inherently viewed as the network design counterpart to related network analysis problems, such as determining strongly connected components. While the design aspect has seen limited progress in distributed settings, some advancements have been made on the analysis front. The work in~\cite{reed2022scalable} introduced an efficient distributed method to compute the strongly connected components of a digraph extended in~\cite{ramos2024distributed} to a distributed approach to classify strongly connected components based on their connectivity within a directed acyclic graph. 

{\color{blue}
The work in~\cite{atman2021distributed} presents distributed algorithms to verify and ensure strong connectivity in directed networks, relying on a maximum consensus algorithm. This method does not require adding the minimum number of edges to ensure strongly connectivity. 
In~\cite{atman2022finite}, the authors extend the previous work to consider digraph augmentation for weakly connected digraphs and disconnected digraphs.  
The algorithm presented by the authors first finds a non-optimal set of edges to make the digraph strongly connected and then (s)elects a virtual leader that aggregates the necessary information to compute (in a centralized manner) a minimal augmenting set of the original graph and broadcasts this information. In contrast, our algorithm is fully distributed, with only a few steps requiring
communication among non-neighboring agents to establish the
necessary additional communication links. Therefore, it builds the optimal solution incrementally without adding spurious edges during its computation.
}

In this paper, our main contribution is the first distributed method for addressing the graph augmentation problem, focusing on determining the minimum number of edges required to render a given weakly-connected graph into a strongly connected. 
%
%
{\color{blue}In particular, the majority of the communication between agents is local (i.e., according to the digraph structure), with only a few steps requiring communication among \mbox{non-neighboring} agents to establish the necessary additional communication links.}  
{\color{blue} Further, we provide a comprehensive empirical analysis
of the algorithm’s performance on random communication
networks (digraph) models, revealing its efficiency and scalability
across different network structures and sizes.

Envisioned applications of our methodology include, but are not limited to, multi-agent systems composed of vehicles interconnected by a communication network are a common approach for surveillance, exploration, and measurement tasks to be accomplished by unmanned and autonomous robotic systems. Missions involving a large number of such agents can adopt a leader/follower approach characterized by having (i) agents that may need to do long-range communication (leaders), and (ii) agents that only require to perform short-range communication (followers)~\cite{lozano2008controllability,ge2018survey,kaminer2007coordinated}.
}

\section{Problem statement}\label{sec:prop_stat}

Consider a multi-agent network, whose communication digraph is given by a \textit{directed graph}, \textit{digraph}, defined as an ordered pair $\mathcal{G} = (\mathcal{V}, \mathcal{E})$, where $\mathcal{V}$ is a nonempty set of nodes, and $\mathcal{E} \subseteq \mathcal{V} \times \mathcal{V}$ is a set of edges. Each edge is an ordered pair representing the accessibility relationship between nodes (agents). In other words, if $u, v \in \mathcal{V}$ and $(u, v) \in \mathcal{E}$, then node $v$ directly accesses information from node $u$.

The (short-range) communication network $\mathcal{G} = (\mathcal{V}, \mathcal{E})$ is \textit{weakly connected}, i.e., for any pair of distinct nodes $u,v\in\mathcal V$ there is a sequence of distinct nodes $(u=v_1, v_2, \ldots, v_k=v)$, such that (s.t.) $\{(v_i, v_{i+1}),(v_{i+1}, v_{i})\} \cap \mathcal{E}\neq\emptyset$ for all $i = 1, \ldots, k-1$ (i.e., if we ignore the directions of the edges, there is a path from $u$ to $v$). 
A \textit{path} in $\mathcal{G} = (\mathcal{V}, \mathcal{E})$ is defined as a sequence of distinct nodes $(v_1, v_2, \ldots, v_k)$, where $(v_i, v_{i+1}) \in \mathcal{E}$ for all $i = 1, \ldots, k-1$. If there is a path from node $u$ to node $v$, then we express this by $u\rightrsquigarrow v$. If such path does not exist, we express this by $u\nrightrsquigarrow v$.

Nonetheless, the multi-agent network seeks to perform consensus-based operations and, possibly, solve distributed optimization problems. Hence, additional (long-range) communication capabilities are required to use to ensure that the communication network becomes \emph{strongly connected}, i.e., for each pair of distinct nodes $u, v\in \mathcal V$ both $u\rightrsquigarrow v$ \underline{and} $v\rightrsquigarrow u$.

{\color{blue}The problem we study in this work is the fully distributed \textit{minimum cardinality digraph augmentation problem} (MCDAP), which can be posed as follows.    
}

\noindent \MCDAP\quad {\em Given a weakly connected digraph $\mathcal G=(\mathcal V,\mathcal E)$, find a set of (long-range, end-to-end -- E2E) communications captured by a set of edges $\mathcal{E}^\ast $, with}
\begin{equation*}\label{eq:prob_stat}
\resizebox{\linewidth}{!}{%
$
\mathcal E^\ast = 
\displaystyle\mathop{\arg\min}_{\mathcal E'\subset \mathcal V\times\mathcal V}\quad|\mathcal E'|
\text{ s.t. } \mathcal G'=(\mathcal V,\mathcal E\cup\mathcal E')\text{ \em is strongly connected.}
$
}
\end{equation*}

\vspace{.5mm}
That is, the overall communication network formed by \mbox{short- and} long-range communications is strongly connected. 

Notice that the problem \MCDAP~can be solved in $\mathcal O(|\mathcal V|+|\mathcal E|)$ (i.e., in polynomial time)~\cite{doi:10.1137/0205044}. Unfortunately, in the context of geographically distributed multi-agent networks, centralized approaches are not feasible. 
Moreover, since \mbox{long-range} communication is extremely costly, the goal is to avoid it and not rely on a centralized solution where every agent communicates directly with every other agent. 
Thus, in this work, we propose a distributed method to solve the~\MCDAP. 



To achieve this goal, we need the following assumptions.

\begin{assumption}\label{rem:1}
    Each agent is assigned a unique identifier (ID). For simplicity, we represent these IDs as a sequence of integers, which do not need to be consecutive.~\hfill$\nabla$
\end{assumption}

{\color{blue}
\begin{assumption}\label{rem:2}
    Each agent is aware of their communication partners IDs (in-neighbors and out-neighbors).~\hfill$\nabla$
\end{assumption}

\begin{assumption}\label{rem:3}
    Each agent, having information of another agent's ID, can establish a connection to that agent.~\hfill$\nabla$
\end{assumption}

The following assumption is also done in the previous literature~\cite{atman2021distributed,atman2022finite,atman2024distributed}.

\begin{assumption}\label{rem:4}
    Communication between nodes occur in a synchronous manner in discrete-time, e.g., by a clock or by the occurrence of external events.~\hfill$\nabla$
\end{assumption}
}

{\color{blue}
It is also important to emphasize the limited information conditions: at the start of the algorithm, agents are not required to have knowledge of other agents' IDs beyond their immediate neighbors, nor do they need to know the total number of agents in the network.}






\section{Distributed Graph Augmentation Protocols to Achieve Strong Connectivity in Multi-Agent Networks}\label{sec:main}

In what follows, we start by reviewing the minimum number of edges required to achieve a solution to~\MCDAP~-- see \Cref{lemma:s-t}. Next, in \Cref{prop:tight} and~\ref{prop:tight_set}, we introduce a characterization of the impact of adding edges towards a possible feasible solution to MCDAP, which we refer to as \emph{tight edges}. With such characterization in mind, we first present the proposed methodology to solve the~\MCDAP~in a distributed fashion (Algorithm~\ref{fig:schematics})\footnote{A centralized implementation of our method can be found on~{\color{blue}\href{https://github.com/xuizy/distributed\_network\_augmentation}{GitHub}}.}. Then, we present its soundness and computational complexity in~\Cref{th:sound} and~\Cref{th:complexity}, respectively.

For an agent $v \in \mathcal{V}$, the set of nodes in the network that $v$ can directly access information from is denoted by $\inneigh{v} = \{u : (u, v) \in \mathcal{E}\}$, the \textit{in-neighbors} of $v$. Conversely, the set of nodes in the network that $u$ can directly send information to is denoted by $\outneigh{u} = \{v : (u, v) \in \mathcal{E}\}$, the \textit{out-neighbors} of~$u$. 

A maximal subset of nodes $\mathcal S\subset \mathcal V$ such that for each pair of distinct nodes $u, v\in \mathcal S$ both $u\rightrsquigarrow v$ \underline{and} $v\rightrsquigarrow u$ is called a \textit{strongly connected component} (SCC). 
Moreover, let us define the set of edges that start in $\mathcal S$ but end outside  $\mathcal S$ as $\mathcal O_{\mathcal S}=\{(u,v)\,:\,u\in \mathcal S\text{ and }v\in\mathcal V\setminus\mathcal S\}$, and the set of edges that start outside $\mathcal S$ and end in $\mathcal S$ as $\mathcal I_{\mathcal S}=\{(u,v)\,:\,u\in \mathcal V\setminus\mathcal S\text{ and }v\in\mathcal S\}$. The SCC $\mathcal S$ is classified as: a source SCC (s-SCC) if $|\mathcal O_{\mathcal S}|\neq 0$ and $|\mathcal I_{\mathcal S}|=0$; a target SCC (t-SCC) if $|\mathcal O_{\mathcal S}|= 0$ and $|\mathcal I_{\mathcal S}|\neq 0$; a mixed SCC (m-SCC) if $|\mathcal O_{\mathcal S}|\neq 0$ and $|\mathcal I_{\mathcal S}|\neq 0$; and an isolated SCC (i-SCC) if $|\mathcal O_{\mathcal S}|=|\mathcal I_{\mathcal S}|=0$. We define $\gamma(\mathcal G)$ as the maximum between the number of s-SCCs and t-SCCs of $\mathcal G$.


\begin{lemma}[\hspace{-.001mm}\cite{eswaran1976augmentation}]\label{lemma:s-t}
    Given a weakly connected digraph $\mathcal G=(\mathcal V, \mathcal E)$ with $\alpha$ s-SCCs and $\beta$ t-SCCs, any solution to~\mbox{\MCDAP} has $\gamma(\mathcal G):=\max\{\alpha, \beta\}$ edges. 
\end{lemma}

We now characterize the impact of introducing the potential edges identified as part of the feasible solution to~\MCDAP, as described in~\Cref{lemma:s-t}. Specifically, given a weakly connected digraph $\mathcal G=(\mathcal V, \mathcal E)$ with $\alpha$ s-SCCs and $\beta$ t-SCCs, an \textit{augmenting edge} $(t,s)\in(\mathcal V\times\mathcal V)\setminus\mathcal E$ is a \textit{tight edge} if the following conditions hold:
    \noindent $(i)\,$ $s$ belongs to an s-SCC; $(ii)\,$ $t$ belongs to a t-SCC;
    $(iii)\,$ either $\alpha=1$ or there exists $s'$ in an s-SCC other than the one of $s$ s.t. $s'\rightrsquigarrow t$;
    $(iv)\,$ either $\beta=1$ or there exists $t'$ in a t-SCC other than the one of $t$ s.t. $s\rightrsquigarrow t'$.
Subsequently, we obtain the following result.

\begin{prop}\label{prop:tight}
    Given a weakly connected digraph $\mathcal G=(\mathcal V, \mathcal E)$ with a tight edge $e$, then $\gamma(\mathcal G+e)=\gamma(\mathcal G)-1$.~\hfill$\triangle$
\end{prop}

\begin{proof}
    Let $\mathcal G=(\mathcal V, \mathcal E)$ with $\alpha$ s-SCCs and $\beta$ t-SCCs and a tight edge $e=(t,s)$. 
In particular, $\alpha,\beta > 0$. 
Let $\alpha',\beta'$ be the number of s-SCCs and t-SCCs of $\mathcal G+e$, respectively. 
We consider the following scenarios: 
    
    Suppose $\alpha,\beta> 1$, then by $(iii)$ and $(iv)$ there exists $s'$ in an s-SCC other than the one of $s$ s.t. $s'\rightrsquigarrow t$ and there exists $t'$ in a t-SCC other than the one of $t$ s.t. $s\rightrsquigarrow t'$. 
    Therefore, in $\mathcal G+e$ there is the path $s'\rightrsquigarrow t\rightrsquigarrow s\rightrsquigarrow t'$. 
    Since $s'$ remains in an s-SCC and $t'$ remains in a t-SCC in $\mathcal G+e$ and $s$ and $t$ now belong to an m-SCC, then $\gamma(\mathcal G+e)=\max\{\alpha', \beta'\}=\max\{\alpha-1, \beta-1\}=\max\{\alpha, \beta\}-1=\gamma(\mathcal G)-1$. 

    Now, suppose that $\alpha=1$ and $\beta>1$. By $(iv)$, there exists $t'$ in a t-SCC other than the one of $t$ s.t. $s\rightrsquigarrow t'$. 
    Thus, in $\mathcal G+e$ there is the path $t\rightrsquigarrow s\rightrsquigarrow t'$. 
    Since $t'$ remains in a t-SCC in $\mathcal G+e$ and $t$ now belongs to an s-SCC, then $\gamma(\mathcal G+e)=\max\{\alpha', \beta'\}=\max\{1, \beta-1\}=\beta-1=\gamma(\mathcal G)-1$. 
    The cases $(\alpha>1,\beta=1)$ and $(\alpha=1,\beta=1)$ are similar.
\end{proof}

We can extend this definition to a set of augmenting edges. 
$\{(t_1,s_1),\ldots,(t_k,s_k)\}$ is a \textit{tight set} if each $(t_i,s_i)$ is tight and the SCCs associated with $t_1,\ldots,t_k,s_1,\ldots,s_k$ are all distinct. Thus, we obtain the following result.

\begin{prop}\label{prop:tight_set}
    Given a weakly connected digraph $\mathcal G=(\mathcal V, \mathcal E)$ with $\alpha$ s-SCCs and $\beta$ t-SCCs and a tight set $\mathcal T=\{(t_1,s_1),\ldots,(t_k,s_k)\}$, then $\gamma(\mathcal G+\mathcal T)=\gamma(\mathcal G)-k$.~\hfill$\triangle$
\end{prop}

\begin{proof}
    The proof follows by induction on the number of elements of $\mathcal T$ by repeatedly applying~\Cref{prop:tight} and noticing that adding a tight edge $(t_i,s_i)$ does not interfere with the tightness of another edge $(t_j,s_j)$ as long as $t,t',s,s'$ belong to different SCCs. 
\end{proof}

With the characterization in~\Cref{prop:tight_set}, we propose to determine the different \mbox{s- and} \mbox{t-SCCs} in a fully distributed fashion in a first phase, followed by different phases entailing a cooperation between representative agents in different \mbox{s- and} \mbox{t-SCCs} towards achieving the characterization in~\Cref{prop:tight_set}, to achieve a tight set. The proposed method can be decomposed into four procedures, which we call phases, interacting as in~Algorithm~\ref{fig:schematics}. 

\begin{alg}[!ht]
    \centering
\begin{tikzpicture}[scale=.45, very thick, every node/.style={scale=0.45}, node distance=1.5cm]

\node (alg1) [io]{
$
\begin{array}{c}
\text{\Large Phase~1}\\\text{ Computation of SCCs}
\end{array}
$};
\node (finish) [io3, right=0.5cm of alg1] {
\Large Finish};
\node (alg4) [io4, below left=.7cm of alg1] {
$
\begin{array}{c}
\text{\Large Phase~4}\\\text{ Selections from targets}\\[-1mm]
\text{ and augmentation}
\end{array}
$};

\node (alg2) [io4,  below=0.55cm of finish] {
$
\begin{array}{c}
\text{\Large Phase~2}\\\text{ Proposals from targets}
\end{array}
$};

\node (alg3) [process, below=1.4cm of alg1] {$
\begin{array}{c}
\text{\Large Phase~3}\\\text{ Selections and proposals}\\[-1mm] 
\text{ from targets}
\end{array}
$};

\draw [arrow, loop left, draw=black!60] (alg1) to (alg1);
\draw [arrow, draw=black!60] (alg1) -- (alg2);
\draw [arrow, dotted, bend left=15, draw=black!60] (alg2) to (alg3);
\draw [arrow, bend right=25, draw=black!60] (alg1) to (alg3);
\draw [arrow, bend right=25, draw=black!60] (alg3) to (alg1);
\draw [arrow, bend left=62, draw=black!60] (alg2) to (alg4);
\draw [arrow, dotted, bend left=15, draw=black!60] (alg3) to (alg4);
\draw [arrow, draw=black!60] (alg4) -- (alg1);
\draw [arrow, draw=black!60] (alg1) -- (finish);


\end{tikzpicture}
    \caption{\color{blue}\footnotesize Schematics of distributed digraph augmentation protocol to address \MCDAP. 
    The scheme starts in the rounded (green) node, ``Phase~\ref{phase:sccs}'', and ends in the diamond (red) node, ``Finish''. 
    Solid arrows represent the possible node transition between phases of the algorithm.
    Dotted arrows represent non-local communication between (source and target) agents.}
    \label{fig:schematics}
\end{alg}

To further describe the algorithm, we require the notation $\sendmsg{u}{v}{i}{m}$ (or, $\send{u}{v}{i}$) to denote that agent $u$ send the message $m$ (or, some message) to agent $v$ in phase $i$, via short-range communication, and $\receivemsg{u}{v}{i}{m}$ (or, $\receive{u}{v}{i}$) to denote that agent $u$ receives the message $m$ (or, some message) from agent $v$ in phase~$i$, via short-range communication. 
If the communication is long-range, we replace the symbols ``$\gg$'' and ``$\ll$'' with ``$\ggg$'' and ``$\lll$'', respectively. 
Additionally, we denote by \textsc{Tarjan} the classical SCCs algorithm by Tarjan~\cite{tarjan1972depth}.

Phase~\ref{phase:sccs} consists of the computations of the SCCs, where an agent $u$ communicates with the out-neighbors until there is no novel information to send. Next,  each agent reconstruct the subdigraph with the information that reached it, compute the subdigraph SCCs and classifies the SCC it belongs to. Depending on the classification, the agent can exit or jump to the start of either Phase~\ref{phase:sccs}, \ref{phase:targetspropose} or \ref{phase:sourcespropose}.

\captionsetup[algorithm]{name=Phase}
\begin{algorithm}[!ht]
\caption{\small Computations of SCCs (by an agent $u$)}\label{phase:sccs}
\begin{algorithmic}[1]
    \scriptsize
    \Statex{\phase{Edge propagation}}
    \State{ $x_u^{(1)}\gets \{(u,v)\,:\,v\in\outneigh{u}\}$}\label{a1:l1}\comment{\color{blue}each agent knows the edges from it to each out-neighbor}
    \State $k\gets 1$
    \While{$x_u^{(k-1)}\neq x_u^{(k)}$}\comment{repeats until no new information}
    \State $\nu_u^{(k)}\gets x_u^{(k)}\setminus x_u^{(k-1)}$\comment{computes the new edges}
    \State $\sendmsg{u}{v}{1}{\nu_u^{(k)}}$ to all $v\in\outneigh{u}$
    \quad and\quad $\receivemsg{u}{v}{1}{\nu_v^{(k)}}$ from all $v\in\inneigh{u}$
    \State $x_u^{(k+1)}\gets x_u^{(k)}\cup\displaystyle\bigcup_{v\in\inneigh{u}}\nu_v^{(k)}$
    \comment{updates the known edges}
    \State $k\gets k+1$
    \EndWhile\label{a1:l9}
    \Statex\phase{SCCs from local knowledge}
    \State $\mathcal S_u\gets\textsc{Tarjan}\left(x_u^{(k)}\right)$\comment{computes the SCCs from the known edges\label{line:tarjan} {\color{blue} a mapping that assigns}} 
    \Statex{\commentContinue{\color{blue}nodes in  $x_u^{(k-1)}$ to the known SCCs, i.e., $\mathcal S_u[v]=i$ if}}
    \Statex{\commentContinue{\color{blue}$v\in x_u^{(k-1)}$ and $v$ is in the $i$th computed SCC}}

{\color{blue}
    \State $\mathcal S_u[u]$ is an s-SCC if $\mathcal A=\{(u, v)\,:\, (u, v) \in x_u^{(k)}, u\notin\mathcal S_u[u]\text{ and }v\in\mathcal S_u[u]\} = \emptyset$
    \comment{\color{blue}it is a source if no node outside the SCC reaches the SCC}
    \State $\mathcal S_u[u]$ is a t-SCC if $\mathcal B=\{(u, v)\,:\, (u, v) \in x_u^{(k)}, u\in\mathcal S_u[u]\text{ and }v\notin\mathcal S_u[u]\} = \emptyset$
    \comment{\color{blue}it is a target if no node outside the SCC is reached by the SCC}
    \State $\mathcal S_u[u]$ is an m-SCC if $\mathcal A \neq \emptyset$ and $\mathcal B \neq \emptyset$
    \comment{\color{blue}it is an m-SCC if nodes outside the SCC reach the SCC and nodes outside the SCC are reached by the SCC}
    }

    \If{$\mathcal S_u[u]$ is an i-SCC\label{a1:11}}\comment{digraph is strongly connected}
        \State{\textbf{exit}}   
    \ElsIf{$\mathcal S_u[u]$ is an s-SCC and $u=\min \mathcal S_u[u]$}\comment{representative is lowest ID agent}
        \State{\textbf{go to} Phase~\ref{phase:sourcespropose}}
    \ElsIf{$\mathcal S_u[u]$ is a t-SCC and $u=\min \mathcal S_u[u]$}\comment{representative is lowest ID agent}
        \State{\textbf{go to} Phase~\ref{phase:targetspropose}}
    \Else\comment{$\mathcal S_u[u]$ is neither an s-SCC nor a t-SCC, or $u$ is not a representative}
        \State{\textbf{go to} Phase~\ref{phase:sccs}\comment{wait until next round}}
    \EndIf\label{a1:19}
\end{algorithmic}
\end{algorithm}

We use an illustrative example to see the different phases of the algorithm. 
Given $\mathcal G=(\mathcal V,\mathcal E)$, we associate with it a \textit{directed acyclic graph} (DAG) representation $\mathcal D=(\mathcal V',\mathcal E')$, where 
$\mathcal V'$is the set of SCCs of $\mathcal G$ and $(s_1,s_2)\in\mathcal E'$ \textit{if and only if} $s_1\neq s_2$ and there is $(u,v)\in\mathcal E$ s.t. $u\in s_1$ and $v\in s_2$. 

\vspace{1mm}
\hrule
\vspace{1mm}
{\footnotesize
\underline{\textit{Illustrative example}:} Consider, as running example, the following network. 
\begin{center}
\begin{tikzpicture}[scale=.4, transform shape,node distance=1.5cm]
\begin{scope}[every node/.style={circle,thick,draw},square/.style={regular polygon,regular polygon sides=4}]
\node (1) at (7.80718,0.) {\Large $1$};
\node (2) at (10.5048,0.) {\Large $2$};
\node (3) at (12.8161,0.) {\Large $3$};
\node (4) at (0.,0.) {\Large $4$};
\node (5) at (2.31048,0.) {\Large $5$};
\node (6) at (5.00628,0.) {\Large $6$};
\end{scope}
\begin{scope}
    \node [fit=(1), c3, draw=c2!90, very thick, inner sep=3pt] {};
    \node [fit=(4) (5), c3, draw=c2!90, very thick, inner sep=3pt] {};
    \node [fit=(6), c3, draw=c1!70, very thick, inner sep=3pt] {};
    \node [fit=(2) (3), c3, draw=c1!70, very thick, inner sep=3pt] {};
\end{scope}
\begin{scope}[>={Stealth[black]},
              every edge/.style={draw=black, thick}]
\path [->] (1) edge node {} (2);
\path [->] (2) edge[bend right=15] node {} (3);
\path [->] (3) edge[bend right=15] node {} (2);
\path [->] (4) edge[bend right=15] node {} (5);
\path [->] (5) edge[bend right=15] node {} (4);
\path [->] (5) edge node {} (6);
\path [->] (1) edge node {} (6);
\end{scope}
\end{tikzpicture}
\end{center}
\vspace{-2.5mm}
}
{\footnotesize After Phase~\ref{phase:sccs}, each agent knows the subgraph of the DAG below consisting of nodes and edges that form paths reaching the agent. 
The yellow nodes are s-SCCs and the blue ones are t-SCCs. Also, the representative agents in each SCC are marked in bold.} 

\vspace{1mm}
\hrule
\vspace{1mm}

Subsequently, in~Phase~\ref{phase:targetspropose}, a representative agent in each \mbox{t-SCC} establishes a connection to a representative agent in each s-SCC that reaches it, and proposes to create a new edge from that t-SCC to the s-SCC. We select as representative agent the one with smallest ID, without loss of generality, as any other agent in the SCC would serve the purpose. 

\begin{algorithm}[!ht]
\caption{\small Proposals from targets}
\label{phase:targetspropose}
\begin{algorithmic}[1]
    \scriptsize
        \State{$\representative{S}_u\gets\emptyset$\comment{\color{blue}set of representatives for the SCCs known by $u$}}
        \For{each s-SCC $ S\in\mathcal S_u$}
                \State{$\representative{S}_u\gets \representative{S}_u\cup\displaystyle\mathop{\arg\min}_{v'\in S}v'$\comment{representative is lowest ID agent}}
            \EndFor
        \State{$\sendlongmsg{u}{s}{2}{ \representative{S}_u}$ to all $s\in\representative{S}_u$\label{step:phase2send}\comment{target proposes to sources}}
        \State{\textbf{go to} Phase~\ref{phase:augmentation}}
\end{algorithmic}
\end{algorithm}

\vspace{1mm}
\hrule
\vspace{1mm}
{\footnotesize
\underline{\textit{Illustrative example}:} In Phase~\ref{phase:targetspropose}, we have the following communications: $(\sendlongmsg{2}{1}{2}{\{1\}})$; $(\sendlongmsg{6}{1}{2}{\{1,4\}})$; and $(\sendlongmsg{6}{4}{2}{\{1,4\}})$.
}
\vspace{1mm}
\hrule
\vspace{1mm}


In~Phase~\ref{phase:sourcespropose}, a representative agent in each \mbox{s-SCC} that received a proposal, selects one of the  proposals corresponding to an edge that, if added, reduces the number of s-SCCs, and proposes to that t-SCC. 
If there is exactly one s-SCC, its representative adds all edges required to make the graph strongly connected. 

\begin{algorithm}[!ht]
\caption{\small Selections and proposals from sources}
\label{phase:sourcespropose}
\begin{algorithmic}[1]
    \scriptsize
        \State{$P_u\gets\{t\,:\,\sendlong{t}{u}{2}\}$\comment{all targets that proposed to $u$}}

        \State{$\receivelongmsg{u}{t}{2}{ \representative{S}_t}$ from all $t\in P_u$\comment{dual of Phase~\ref{phase:targetspropose}, \cref{step:phase2send}}}
        \If{there exists $t\in P_u$ s.t. $|\representative{S}_{t}|>1$}\comment{prefer targets reachable by multiple sources}
            \State{\textbf{select} some $t\in P_u$ s.t. $|\representative{S}_{t}|>1$}
            \State{$\sendmsg{u}{t}{3}{P_u}$ \label{step:phase3send}\comment{source proposes to target}}
        \Else\comment{there is exactly one s-SCC}
            \State{$\mathcal E^\ast\gets\mathcal E^\ast\cup\{(t,u)\,:\,\text{for all }t\in P_u\}$\comment{add edges to the digraph}}\label{a3:l7}
        \EndIf
    \State{\textbf{go to} Phase~\ref{phase:sccs}\comment{wait until next round}}
\end{algorithmic}
\end{algorithm}
\vspace{1mm}
\hrule
\vspace{1mm}
{\footnotesize
\underline{\textit{Illustrative example}:} In Phase~\ref{phase:sourcespropose}, we have the following received communications: $(\receivelongmsg{1}{2}{2}{\{1\}})$; $(\receivelongmsg{1}{6}{2}{\{1,4\}})$; and $(\receivelongmsg{4}{6}{2}{\{1,4\}})$. 
Moreover, we have the following sent communications: $(\sendmsg{1}{6}{3}{\{2,6\}})$; and $(\sendmsg{4}{6}{3}{\{6\}})$.
}
\vspace{1mm}
\hrule
\vspace{1mm}

Finally, in~Phase~\ref{phase:augmentation}, a representative agent in each t-SCC that received a proposal, accepts the proposal corresponding to an edge that, if added, reduces the number t-SCCs, and creates the new edge from the representative agent in that t-SCC to the representative one of the selected s-SCC. 
If there is exactly one t-SCC, its representative adds all edges required to make the graph strongly connected. 
\vspace{1mm}
\hrule
\vspace{1mm}
{\footnotesize
\underline{\textit{Illustrative example}:} In Phase~\ref{phase:augmentation}, the representative agent $6$ selects the representative $1$ and, after one round, the network becomes the following. 
\vspace{-.5mm}
\begin{center}
\begin{tikzpicture}[scale=.4, transform shape,node distance=1.5cm]
\begin{scope}[every node/.style={circle,thick,draw},square/.style={regular polygon,regular polygon sides=4}]
\node (1) at (7.80718,0.) {\Large $1$};
\node (2) at (10.5048,0.) {\Large $2$};
\node (3) at (12.8161,0.) {\Large $3$};
\node (4) at (0.,0.) {\Large $4$};
\node (5) at (2.31048,0.) {\Large $5$};
\node (6) at (5.00628,0.) {\Large $6$};
\end{scope}
\begin{scope}[>={Stealth[black]},
              every edge/.style={draw=black, thick}]
\path [->] (1) edge node {} (2);
\path [->] (2) edge[bend right=15] node {} (3);
\path [->] (3) edge[bend right=15] node {} (2);
\path [->] (4) edge[bend right=15] node {} (5);
\path [->] (5) edge[bend right=15] node {} (4);
\path [->] (5) edge node {} (6);
\path [->] (1) edge[bend right=15] node {} (6);
\end{scope}
\begin{scope}[>={Stealth[cadmiumred]}, every edge/.style={draw=cadmiumred, thick}]
\path [->] (6) edge[bend right=15] node[fill=gray!20, text=black, draw=black, thin, solid] {\Large 1} (1);
\end{scope}
\end{tikzpicture}
\end{center}
\vspace{-3mm}
Note that in the new network, instead of having $\alpha=\beta=2$ (two s-SCCs and 2 t-SCCs), we now have $\alpha=\beta=1$. 
}
\vspace{1mm}
\hrule
\vspace{1mm}

\begin{algorithm}[!ht]
\caption{\small Selections from targets and augmentation}
\label{phase:augmentation}
\begin{algorithmic}[1]
    \scriptsize
        \State{$P_u\gets\{s\,:\,\send{s}{u}{3}\}$\comment{all sources that proposed to $u$}}
        \If{$P_u=\emptyset$}
            \State{\textbf{go to} Phase~\ref{phase:sccs}\comment{wait until next round}}
        \EndIf
        \State{$\receivemsg{u}{s}{3}{ P_s
}$ from all $s\in P_u$\comment{dual of Phase~\ref{phase:sourcespropose}, \cref{step:phase3send}}}
        \If{there exists $s\in P_u$ s.t. $|P_{s}|>1$}\comment{prefer sources reaching multiple targets}
            \State{\textbf{select} some $s\in P_u$ s.t. $|P_{s}|>1$}
            \State{$\mathcal E^\ast\gets\mathcal E^\ast\cup(u,s)$\comment{add edge to the digraph}}
        \Else\comment{there is exactly one t-SCC}
            \State{$\mathcal E^\ast\gets\mathcal E^\ast\cup\{(u,s)\,:\,\text{for all }s\in P_s\}$\comment{add edges to the digraph}}\label{a4:l10}
        \EndIf
    \State{\textbf{go to} Phase~\ref{phase:sccs}\comment{wait until next round}}
    \end{algorithmic}
\end{algorithm}


\begin{theorem}\label{th:sound}
    The scheme in~Algorithm~\ref{fig:schematics} is sound, i.e.~ it computes a solution to~\MCDAP~in a distributed fashion.~\hfill$\circ$ 
\end{theorem}

\begin{proof}
    Consider a digraph $\mathcal G=(\mathcal V, \mathcal E)$ for which we apply~Algorithm~\ref{fig:schematics}. 
    We define a (distributed) round of the scheme in~Algorithm~\ref{fig:schematics} as a sequence of steps performed by all agents until they all return to the beginning of~Phase~\ref{phase:sccs}. 
    
    First, we claim that in each round there is at least one edge added to the digraph. 
    Suppose that $\mathcal G$ is not strongly connected. Then, it has at least one s-SCC and one t-SCC. 
    Thus, there is at least a representative of a t-SCC that executes~Phase~\ref{phase:targetspropose} proposing to at least one representative~$s$ of an s-SCC. Conversely, there is at least one representative of an s-SCC that receives at least one proposal and either augments the graph or proposes to some t-SCC representative. In the later case, there is at least a representative of a t-SCC receiving such a proposal and augmenting the network. 

    Second, we claim that each of the aforementioned edges, $(t,s)$, is tight. 
    There are three cases: $(i)$ $| S_u^\star|=1$ -- the case where there is an s-SCC $s$ that received proposals from  one or more t-SCCs, all of each are reachable only by that s-SCC. Since the digraph is weakly connected, then that is the only s-SCC ($\alpha=1$). Thus, all the added edges of the form $(t,s)$ with $t\in{S}_u^\star$ are tight. 
    $(ii)$ $| S_t^\ast|>1$ and $| P_s|=1$ -- this is the case where there is a t-SCC $t$ that received proposals from multiple s-SCCs, all of each reach only that t-SCC. Thus, all the added edges of the form $(t,s)$ with $s\in\mathcal{S}_t^\star$ are tight. $(iii)$ $| S_t^\star|>1$ and $|P_s|>1$ -- this is the case where an edge $(t,s)$ is added. Since $| S_t^\star|>1$, there exists $s'$ in an s-SCC other than the one of $s$ s.t. $s'\rightrsquigarrow t$ and, since $|P_s|>1$, there exists $t'$ in a t-SCC other than the one of $t$ s.t. $s\rightrsquigarrow t'$. Hence, the edge $(t,s)$ is tight. Moreover, each s-SCC proposes at most one edge in~Phase~\ref{phase:sourcespropose} and, subsequently, each t-SCC selects at most one edge for augmentation in~Phase~\ref{phase:augmentation}. Therefore, all endpoints $t_i,s_i$ of edges chosen in this manner belong to distinct SCCs, that is, the set of augmenting edges (in the given round) is tight.

    Finally, let $\mathcal G_0, \mathcal G_1,\ldots, \mathcal G_m$ be the digraphs at the end of each round. At any round $i$, the decrease $\gamma(\mathcal G_{i-1}) - \gamma(\mathcal G_i)$ is exactly the number of augmented edges, by~\Cref{prop:tight_set} and the previous paragraph. 
    When the scheme terminates, the digraph is strongly connected and, therefore, $\alpha_m=\beta_m=\gamma(\mathcal G_m)=0$. Thus, the set of edges produced by the algorithm has cardinality $\gamma(G_0)-\gamma(G_1)+\ldots+\gamma(G_{k-1})-\gamma(G_k)=\gamma(G_0)$. Thus, by~\Cref{lemma:s-t}, it is a solution to~\MCDAP. 
\end{proof}
\vspace{1mm}

    In what follows, we consider the computational complexity analysis of each phase, but we do not consider the steps that call other phases. The computational complexity of the combined phases in Algorithm~\ref{fig:schematics} is relegated to~\Cref{th:complexity}.

\begin{prop}\label{prop:alg1_complexity}
The time-complexity of Phase~\ref{phase:sccs} for an agent $u$ in a network $\mathcal G=(\mathcal V,\mathcal E)$ is $\mathcal O(\max\{|\inneigh{u}||\mathcal E|, |\mathcal V|+|\mathcal E|\})$.~\hfill$\triangle$
\end{prop}

\begin{proof}
    First, we calculate the time-complexity for agent $u$ in~Phase~\ref{phase:sccs}, \cref{a1:l1} --~\cref{a1:l9}. In the worst case, agent $u$ receives all the edges from all of its in-neighbors. Therefore, the time-complexity is $\mathcal O(|\inneigh{u}||\mathcal E|)$. 
    Subsequently, agent $u$, in~\cref{line:tarjan}, computes the SCCs using the Tarjan's algorithm~\cite{tarjan1972depth}, with time-complexity of $\mathcal O(|\mathcal V|+|\mathcal E|)$. 
    Finally, in~\cref{a1:11} --~\cref{a1:19}, the agent classifies the SCC it belongs to, in $\mathcal O(|\mathcal E|)$. 
    Hence, the total \mbox{time-complexity} is the maximum of the above quantities, i.e., $\mathcal O(\max\{|\inneigh{u}||\mathcal E|, |\mathcal V|+|\mathcal E|\})$.
\end{proof}

\begin{prop}\label{prop:alg2_complexity}
The time-complexity of Phase~\ref{phase:targetspropose}  for agent $u$ in $\mathcal G=(\mathcal V,\mathcal E)$ with $\alpha$ s-SCCs and $\beta$ t-SCCs is $\mathcal O(\alpha).$\hfill$\triangle$
\end{prop}

\begin{proof}
    It is easy to check that, in the worst case, agent $u$ proposes to all $\alpha$ s-SCCs, yielding $\mathcal O(\alpha)$. 
\end{proof}

\begin{prop}\label{prop:alg3_complexity}
The time-complexity of Phase~\ref{phase:sourcespropose}  for agent $u$ in $\mathcal G=(\mathcal V,\mathcal E)$ with $\alpha$ s-SCCs and $\beta$ t-SCCs is $\mathcal O(\beta).$\hfill$\triangle$
\end{prop}

\begin{proof}
    It is easy to check that, in the worst case, agent $u$ received proposals from all $\beta$ t-SCCs and selects one by linear search, yielding $\mathcal O(\beta)$. 
\end{proof}

\begin{prop}\label{prop:alg4_complexity}
The time-complexity of Phase~\ref{phase:augmentation}  for agent $u$ in $\mathcal G=(\mathcal V,\mathcal E)$ with $\alpha$ s-SCCs and $\beta$ t-SCCs is $\mathcal O(\alpha).$\hfill$\triangle$
\end{prop}

\begin{proof}
    Similarly, in the worst case, agent $u$ received proposals from all $\alpha$ s-SCCs and selects one, i.e., $\mathcal O(\alpha)$.  
\end{proof}

Next, we present the overall time-complexity of our solution.

\begin{theorem}\label{th:complexity}
    The time-complexity of  Algorithm~\ref{fig:schematics}  for agent $u$ in $\mathcal G=(\mathcal V,\mathcal E)$ with $\alpha$ s-SCCs and $\beta$ t-SCCs is $\mathcal O(|\mathcal V||\mathcal E|\min\{\alpha,\beta\}).$\hfill$\circ$
\end{theorem}

\begin{proof}
    By~\Cref{th:sound}, we know that we need to add, at most, $\max\{\alpha,\beta\}$ edges to the digraph to get a solution. Notwithstanding, from~\cref{a3:l7} of Phase~\ref{phase:sourcespropose} and~\cref{a4:l10} of Phase~\ref{phase:augmentation}, when $\alpha$ decrease to $1$ or $\beta$ decreases to $1$, in a single round the remaining edges are added, yielding $\mathcal O(\min\{\alpha,\beta\})$ rounds. 
    So, each round of Algorithm~\ref{fig:schematics} costs, at most, $\mathcal O(\max\{|\inneigh{u}||\mathcal E|, |\mathcal V|+|\mathcal E|\}+\alpha+\beta+\alpha)=\mathcal O(|\mathcal V||\mathcal E|)$, by~\Cref{prop:alg1_complexity} --~\ref{prop:alg4_complexity} and the fact that $\alpha,\beta\leq|\mathcal V|$. Hence, the total time-complexity is $\mathcal O(|\mathcal V||\mathcal E|\max\{\alpha,\beta\}).$
\end{proof}

In fact, in the particular cases where there is a single s-SCC or a single t-SCC, it is easy to check that in~\cref{a3:l7} of Phase~\ref{phase:sourcespropose} or in~\cref{a4:l10} of Phase~\ref{phase:augmentation}, respectively, all the required edges are added within a single round, entailing the following.

\begin{corollary}\label{cor:complexity}
    The time-complexity of Algorithm~\ref{fig:schematics}, for agent $u$ in $\mathcal G=(\mathcal V,\mathcal E)$ with $\alpha=1$ or $\beta=1$ is $\mathcal O(|\mathcal V||\mathcal E|)$.~\hfill$\sharp$
\end{corollary}





 \section{Illustrative Examples}\label{sec:ill}
{\color{blue}
We present  pedagogical examples (\Cref{sub:sim}), and a study of the our algorithm on random graphs (\Cref{sub:disc}).

\subsection{Pedagogical Examples}\label{sub:sim}

In what follows, each box of yellow nodes represents the agents in an s-SCC, each blue box the agents in a t-SCC, and each green box the agents in an m-SCC. Also, colors and labels are as in Fig.~\ref{fig:exp:running}.

The first example is the previous running example, depicted in~\Cref{fig:exp:running}. 
The network for which we address~\MCDAP~is represented by the digraph with black edges in~\Cref{fig:exp:running}.  The solution obtained with Algorithm~\ref{fig:schematics} is $\mathcal E^\ast=\{(6,1),(2,4)\}$, depicted by the red edges in~\Cref{fig:exp:running}. 
Moreover, the round where each new edge is created is annotated in the red edges.

\begin{figure}[!ht]
\centering
{
\begin{tikzpicture}[scale=.38, transform shape,node distance=1.5cm]
\begin{scope}[every node/.style={circle,thick,draw},square/.style={regular polygon,regular polygon sides=4}]
\node (1) at (7.80718,0.) {\Large $1$};
\node (2) at (10.5048,0.) {\Large $2$};
\node (3) at (12.8161,0.) {\Large $3$};
\node (4) at (0.,0.) {\Large $4$};
\node (5) at (2.31048,0.) {\Large $5$};
\node (6) at (5.00628,0.) {\Large $6$};
\end{scope}
\begin{scope}
    \node [fit=(1), c3, draw=c2!90, very thick, inner sep=2pt] {};
    \node [fit=(4) (5), c3, draw=c2!90, very thick, inner sep=2pt] {};
    \node [fit=(6), c3, draw=c1!90, very thick, inner sep=2pt] {};
    \node [fit=(2) (3), c3, draw=c1!90, very thick, inner sep=2pt] {};
\end{scope}
\begin{scope}[>={Stealth[black]},
              every edge/.style={draw=black, thick}]
\path [->] (1) edge node {} (2);
\path [->] (2) edge[bend right=15] node {} (3);
\path [->] (3) edge[bend right=15] node {} (2);
\path [->] (4) edge[bend right=15] node {} (5);
\path [->] (5) edge[bend right=15] node {} (4);
\path [->] (5) edge node {} (6);
\path [->] (1) edge[bend right=15] node {} (6);
\end{scope}
\begin{scope}[>={Stealth[cadmiumred]}, every edge/.style={draw=cadmiumred, thick}]
\path [->] (6) edge[bend right=15] node[fill=gray!20, text=black, draw=black, thin, solid] {\Large 1} (1);
\path [->] (2) edge[bend right=22]  node[fill=gray!20, text=black, draw=black, thin, solid] {\Large 2} (4);
\end{scope}
\end{tikzpicture}
}
\caption{The digraph with black edges is the one to be augmented (\MCDAP), which decomposition in SCCs is represented by boxes of nodes. The red edges represent the augmentation computed using the scheme in~Algorithm~\ref{fig:schematics}. The edges' labels are the round in which the edge is created.}
\label{fig:exp:running}
\end{figure}
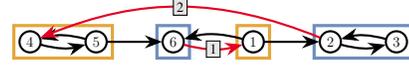

Following the same setup, the second example is depicted in~\Cref{fig:exp1}. 
Although one can always design a solution only connecting \mbox{t-SCCs} to \mbox{s-SCCs}, our proposal is flexible and can produce other solutions, as this one. 
 

\vspace{-3mm}
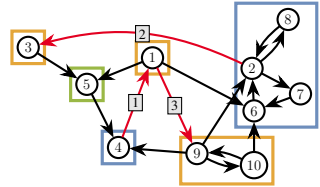
\begin{figure}[!ht]
  \begin{minipage}[c]{0.56\columnwidth} 
    \centering
    \begin{tikzpicture}[scale=.38, transform shape,node distance=1.5cm]
\begin{scope}[every node/.style={circle,thick,draw},square/.style={regular polygon,regular polygon sides=4}]
\node (1) at (4.34864,4.25775) {\Large $1$};
\node (2) at (7.8313,3.88845) {\Large $2$};
\node (3) at (0.,4.62387) {\Large $3$};
\node (4) at (3.1583,1.11785) {\Large $4$};
\node (5) at (2.05192,3.32892) {\Large $5$};
\node (6) at (7.9,2.4) {\Large $6$};
\node (7) at (9.52516,2.99385) {\Large $7$};
\node (8) at (9.09311,6.-.4) {\Large $8$};
\node (9) at (5.90349,1.00426-.1) {\Large $9$};
\node (10) at (7.9,0.6-.1) {\Large $10$};
\end{scope}
\begin{scope}
    \node [fit=(1), c3, draw=c2!90, very thick, inner sep=2pt] {};
    \node [fit=(3), c3, draw=c2!90, very thick, inner sep=2pt] {};
    \node [fit=(9) (10), c3, draw=c2!90, very thick, inner sep=2pt] {};
    \node [fit=(4), c3, draw=c1!90, very thick, inner sep=2pt] {};
    \node [fit=(8) (2) (6) (7), c3, draw=c1!90, very thick, inner sep=2pt] {};
    \node [fit=(5), c3, draw=c3!90, very thick, inner sep=2pt] {};
\end{scope}
\begin{scope}[>={Stealth[black]},
              every edge/.style={draw=black, thick}]
\path [->] (1) edge node {} (5);
\path [->] (1) edge node {} (6);
\path [->] (2) edge node {} (7);
\path [->] (2) edge[bend right=15] node {} (8);
\path [->] (3) edge node {} (5);
\path [->] (5) edge node {} (4);
\path [->] (6) edge node {} (2);
\path [->] (7) edge node {} (6);
\path [->] (8) edge[bend right=15] node {} (2);
\path [->] (9) edge node {} (2);
\path [->] (9) edge node {} (4);
\path [->] (9) edge[bend right=15] node {} (10);
\path [->] (10) edge node {} (6);
\path [->] (10) edge[bend right=15] node {} (9);
\end{scope}
\begin{scope}[>={Stealth[cadmiumred]}, every edge/.style={draw=cadmiumred, thick}]
\path [->] (1) edge node[fill=gray!20, text=black, draw=black, thin, solid] {\Large 3} (9);
\path [->] (2) edge[bend right=20] node[fill=gray!20, text=black, draw=black, thin, solid] {\Large 2} (3);
\path [->] (4) edge node[fill=gray!20, text=black, draw=black, thin, solid] {\Large 1} (1);
\end{scope}
\end{tikzpicture}
  \end{minipage}%
  \hspace{0.01\columnwidth} 
  \begin{minipage}[c]{0.41\columnwidth}
    \centering
    \caption{
       The digraph with black edges in~(a) is the digraph to be augmented (\MCDAP), which decomposition in SCCs is represented by boxes of nodes. $\mathcal E^\ast=\{(4,1)(2,3),(1,9)\}$.
    } \label{fig:exp1}
  \end{minipage}
\end{figure}
\vspace{-3mm}



In the Fig.~\ref{fig:exp2} example, 2 edges were created in the first round. 

\begin{figure}[!ht]
    \centering
    \begin{tikzpicture}[scale=.38, transform shape,node distance=1.5cm]
\begin{scope}[every node/.style={circle,thick,draw},square/.style={regular polygon,regular polygon sides=4}]
\node (1) at (14.9233,3.7034) {\Large $1$};
\node (2) at (17.323,3.83078) {\Large $2$};
\node (3) at (7.47641,7.06299-1.2) {\Large $3$};
\node (4) at (7.42275,4.12663-.2) {\Large $4$};
\node (5) at (0.0488491-1.,5.25912+.2) {\Large $5$};
\node (6) at (0.995835+.5-.9,6.28858-.3) {\Large $6$};
\node (7) at (0.-.9,3.72284) {\Large $7$};
\node (8) at (2.78184-.9,4.06325) {\Large $8$};
\node (9) at (11.7964,5.46256-.6) {\Large $9$};
\node (10) at (9.89909,3.70164) {\Large $10$};
\node (11) at (12.4444,3.32596) {\Large $11$};
\node (12) at (3.95642-.3,0.347246+2.2) {\Large $12$};
\node (14) at (5.20421+.3,0.+2.3) {\Large $14$};
\node (13) at (4.92066,2.83159+1.5) {\Large $13$};
\end{scope}
\begin{scope}
    \node [fit=(3) (4), c3, draw=c2!90, very thick, inner sep=2pt] {};
    \node [fit=(1) (2), c3, draw=c2!90, very thick, inner sep=2pt] {};
    \node [fit=(5) (6) (8) (7), c3, draw=c2!90, very thick, inner sep=2pt] {};
    \node [fit=(12) (13) (14), c3, draw=c1!90, very thick, inner sep=2pt] {};
    \node [fit=(9) (10) (11), c3, draw=c1!90, very thick, inner sep=2pt] {};
\end{scope}
\begin{scope}[>={Stealth[black]},
              every edge/.style={draw=black, thick}]
\path [->] (1) edge[bend right=15] node {} (2);
\path [->] (2) edge[bend right=15] node {} (1);
\path [->] (3) edge[bend right=15] node {} (4);
\path [->] (4) edge[bend right=15] node {} (3);
\path [->] (5) edge[bend right=15] node {} (6);
\path [->] (5) edge node {} (7);
\path [->] (5) edge node {} (8);
\path [->] (6) edge[bend right=15] node {} (5);
\path [->] (6) edge[bend right=15] node {} (8);
\path [->] (7) edge node {} (6);
\path [->] (8) edge[bend right=15] node {} (6);
\path [->] (8) edge node {} (7);
\path [->] (9) edge node {} (10);
\path [->] (10) edge node {} (11);
\path [->] (11) edge node {} (9);
\path [->] (12) edge node {} (14);
\path [->] (13) edge node {} (12);
\path [->] (14) edge node {} (13);
\path [->] (1) edge node {} (11);
\path [->] (4) edge node {} (10);
\path [->] (4) edge node {} (13);
\path [->] (8) edge node {} (13);
\end{scope}
\begin{scope}[>={Stealth[cadmiumred]}, every edge/.style={draw=cadmiumred, thick}]
\path [->] (12) edge[bend left=25] node[fill=gray!20, text=black, draw=black, thin, solid] {\Large 1} (5);
\path [->] (9) edge[bend right=25] node[fill=gray!20, text=black, draw=black, thin, solid] {\Large 1} (3);
\path [->] (5) edge[bend right=28] node[fill=gray!20, text=black, draw=black, thin, solid] {\Large 2} (1);
\end{scope}
\end{tikzpicture}
    \caption{The digraph with black edges is the one to be augmented (\MCDAP), which decomposition in SCCs is represented by boxes of nodes. $\mathcal E^\ast=\{(9,1)(12,5),(5,1)\}$.}
    \label{fig:exp2}
\end{figure}
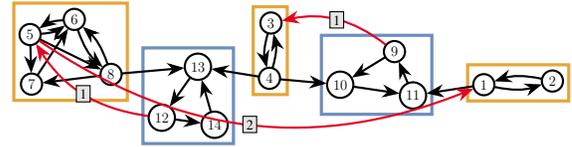

\subsection{Random Graphs}\label{sub:disc}



We analyze the average number of rounds required by Algorithm~\ref{fig:schematics} on randomly generated digraphs. For each set of parameters, we generate 30 random weakly connected digraphs using four methods: (i) Erdős-Rényi, Bernoulli, Erdős-Rényi uniform, and random DAG generation~\cite{drobyshevskiy2019random}. In the Erdős-Rényi and Bernoulli models, with $n$ nodes and edge probability $p$, edges are added or removed from an empty or complete graph, respectively. In the Erdős-Rényi uniform and random DAG models, $n$ is the number of nodes and $m$ is the number of edges.

Results show that, except for random DAGs, the number of rounds (i.e., $\min\{\alpha, \beta\}$) remains approximately the same as parameters increase. In contrast, random DAGs, with no cycles, lead to more SCCs,  increasing the number of rounds.

\begin{figure}[!ht]
    \centering
    \subfigure[{\tiny Random \"Erdos-Rényi ($xx$: $n$, $yy$: $p$).}]{
    \includegraphics[width=.465\columnwidth]{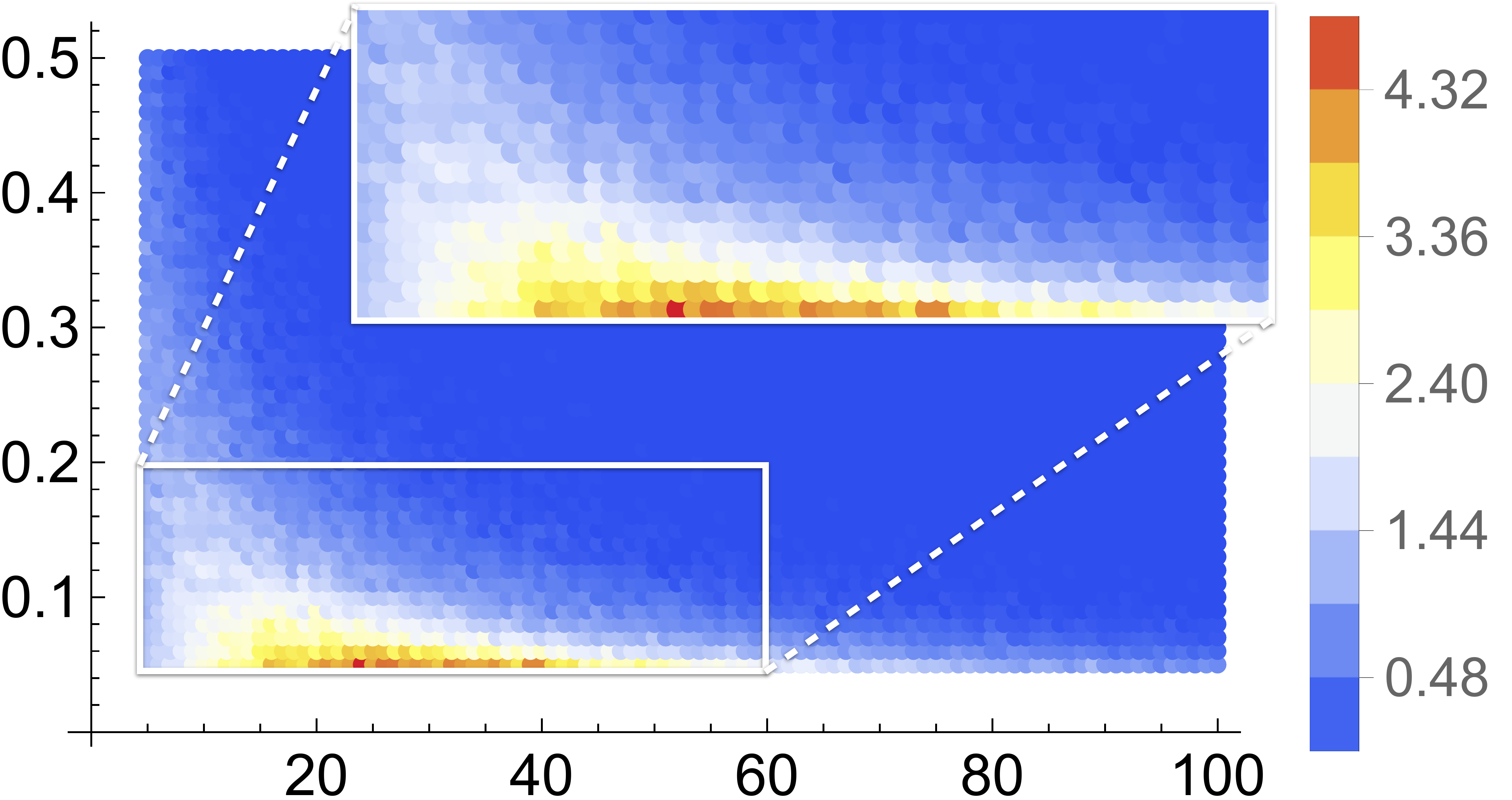}
    }    
    \subfigure[\tiny Random Bernoulli ($xx$: $n$, $yy$: $p$).]{
    \includegraphics[width=.465\columnwidth]{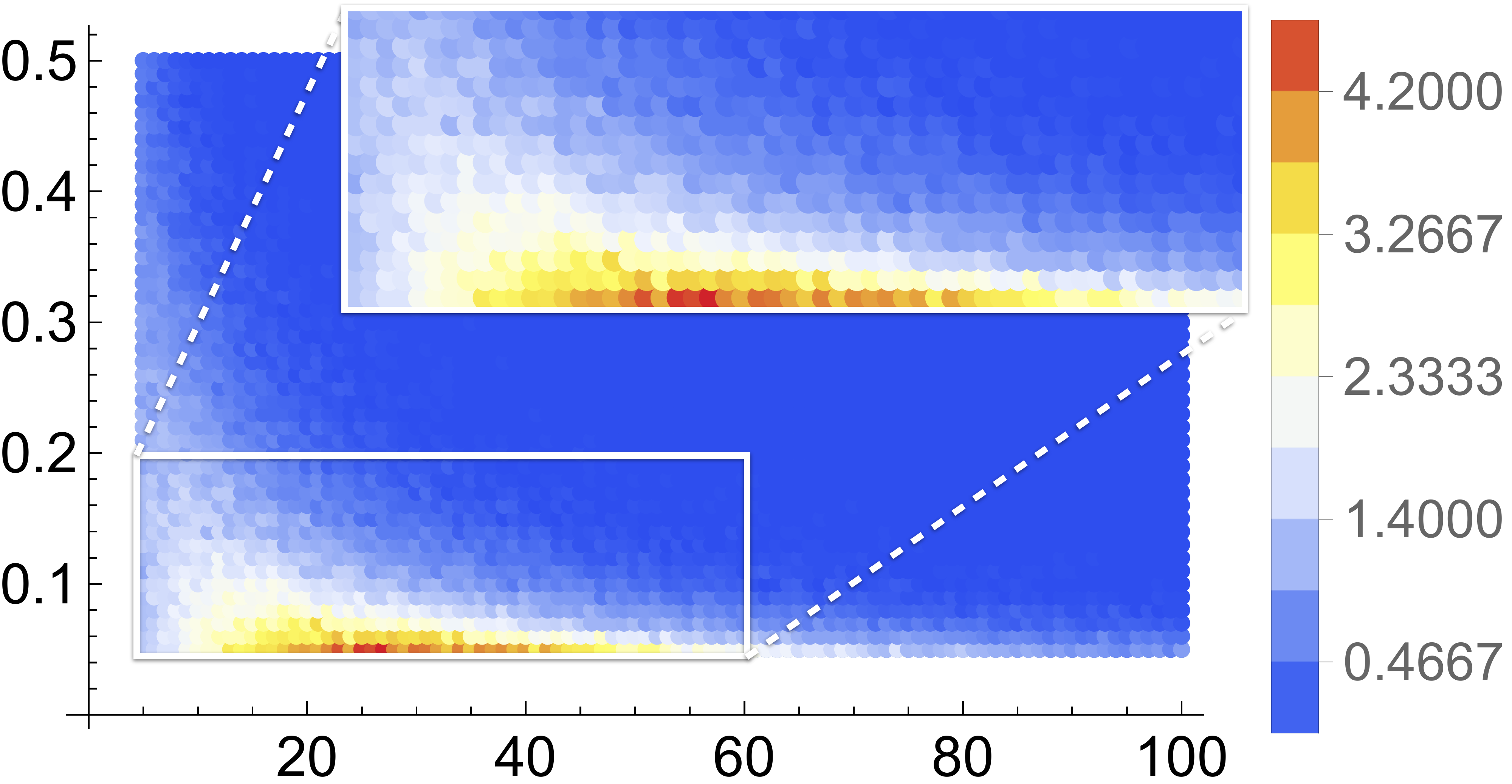}
    }
    \subfigure[\tiny Random \"Erdos-Rényi uniform ($xx$: $n$, $yy$: $m$).]{
    \includegraphics[width=.465\columnwidth]{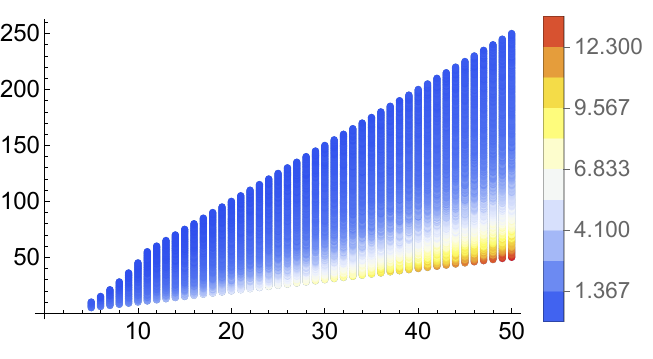}
    }
    \subfigure[\tiny Random DAGs ($xx$: $n$, $yy$: $m$).]{
    \includegraphics[width=.465\columnwidth]{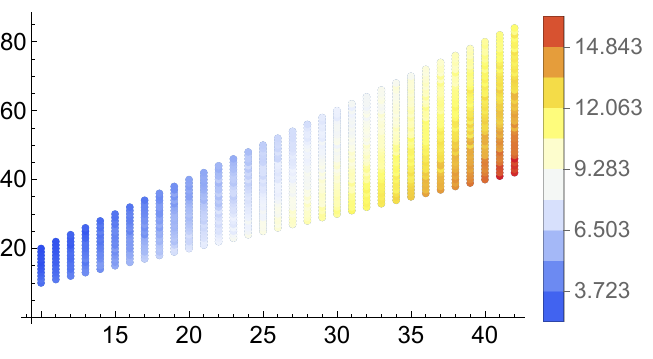}
    }
    \caption{Average number of rounds of Algorithm~\ref{fig:schematics} on random digraphs, using different random generation models.}
    \label{fig:enter-label}
\end{figure}





}

\section{Conclusions}\label{sec:conclusion}

This paper introduces the first distributed method for solving a communication network augmentation problem,  where agents efficiently identify a minimum number of edges required to render a strongly connected communication network. 
Our empirical analysis of random communication networks further validates the algorithm's efficiency, demonstrating its robust performance across various network structures and sizes. 

Future work presents several exciting avenues for exploration. Key research directions include addressing challenges in dynamic networks where topology evolves over time, developing cost-aware augmentation strategies that consider edge creation costs, and enhancing network resilience against node or edge failures. 


{\footnotesize 
\bibliographystyle{IEEEtran}
\bibliography{IEEEabrv,mybibfile.bib}

\begin{thebibliography}{10}
\providecommand{\url}[1]{#1}
\csname url@samestyle\endcsname
\providecommand{\newblock}{\relax}
\providecommand{\bibinfo}[2]{#2}
\providecommand{\BIBentrySTDinterwordspacing}{\spaceskip=0pt\relax}
\providecommand{\BIBentryALTinterwordstretchfactor}{4}
\providecommand{\BIBentryALTinterwordspacing}{\spaceskip=\fontdimen2\font plus
\BIBentryALTinterwordstretchfactor\fontdimen3\font minus
  \fontdimen4\font\relax}
\providecommand{\BIBforeignlanguage}[2]{{%
\expandafter\ifx\csname l@#1\endcsname\relax
\typeout{** WARNING: IEEEtran.bst: No hyphenation pattern has been}%
\typeout{** loaded for the language `#1'. Using the pattern for}%
\typeout{** the default language instead.}%
\else
\language=\csname l@#1\endcsname
\fi
#2}}
\providecommand{\BIBdecl}{\relax}
\BIBdecl

\bibitem{garin2010survey}
F.~Garin and L.~Schenato, ``A survey on distributed estimation and control
  applications using linear consensus algorithms,'' in \emph{Networked control
  systems}.\hskip 1em plus 0.5em minus 0.4em\relax Springer, 2010, pp. 75--107.

\bibitem{ramossilvestresIJoC}
G.~Ramos, D.~Silvestre, and C.~Silvestre, ``General resilient consensus
  algorithms,'' \emph{International Journal of Control}, vol.~95, no.~6, pp.
  1482--1496, 2022.

\bibitem{ramos2021distributed}
G.~Ramos, A.~P. Aguiar, S.~Kar, and S.~Pequito, ``Privacy preserving average
  consensus through network augmentation,'' \emph{IEEE Transactions on
  Automatic Control}, pp. 1--13, 2024.

\bibitem{amirkhani2022consensus}
A.~Amirkhani and A.~H. Barshooi, ``Consensus in multi-agent systems: a
  review,'' \emph{Artificial Intelligence Review}, vol.~55, no.~5, pp.
  3897--3935, 2022.

\bibitem{ramos2023designing}
G.~Ramos and S.~Pequito, ``Designing communication networks for discrete-time
  consensus for performance and privacy guarantees,'' \emph{Systems \& Control
  Letters}, vol. 180, p. 105608, 2023.

\bibitem{rabbat2004distributed}
M.~Rabbat and R.~Nowak, ``Distributed optimization in sensor networks,'' in
  \emph{Proceedings of the 3rd international symposium on Information
  processing in sensor networks}, 2004, pp. 20--27.

\bibitem{huang2015differentially}
Z.~Huang, S.~Mitra, and N.~Vaidya, ``Differentially private distributed
  optimization,'' in \emph{Proceedings of the 16th International Conference on
  Distributed Computing and Networking}, 2015, pp. 1--10.

\bibitem{molzahn2017survey}
D.~K. Molzahn, F.~D{\"o}rfler, H.~Sandberg, S.~H. Low, S.~Chakrabarti,
  R.~Baldick, and J.~Lavaei, ``A survey of distributed optimization and control
  algorithms for electric power systems,'' \emph{IEEE Transactions on Smart
  Grid}, vol.~8, no.~6, pp. 2941--2962, 2017.

\bibitem{nedic2018distributed}
A.~Nedi{\'c} and J.~Liu, ``Distributed optimization for control,'' \emph{Annual
  Review of Control, Robotics, and Autonomous Systems}, vol.~1, no.~1, pp.
  77--103, 2018.

\bibitem{yang2019survey}
T.~Yang, X.~Yi, J.~Wu, Y.~Yuan, D.~Wu, Z.~Meng, Y.~Hong, H.~Wang, Z.~Lin, and
  K.~H. Johansson, ``A survey of distributed optimization,'' \emph{Annual
  Reviews in Control}, vol.~47, pp. 278--305, 2019.

\bibitem{chlebus2002deterministic}
B.~S. Chlebus, L.~Gasieniec, A.~Gibbons, A.~Pelc, and W.~Rytter,
  ``Deterministic broadcasting in ad hoc radio networks,'' \emph{Distributed
  computing}, vol.~15, no.~1, pp. 27--38, 2002.

\bibitem{li2006construction}
Y.~Li, M.~T. Thai, F.~Wang, and D.-Z. Du, ``On the construction of a strongly
  connected broadcast arborescence with bounded transmission delay,''
  \emph{IEEE Transactions on mobile computing}, vol.~5, no.~10, pp. 1460--1470,
  2006.

\bibitem{bodlaender2013connectivity}
M.~H. Bodlaender, M.~M. Halld{\'o}rsson, and P.~Mitra, ``Connectivity and
  aggregation in multihop wireless networks,'' in \emph{Proceedings of the 2013
  ACM Symposium on Principles of distributed computing}, 2013, pp. 355--364.

\bibitem{eswaran1976augmentation}
K.~P. Eswaran and R.~E. Tarjan, ``Augmentation problems,'' \emph{SIAM Journal
  on Computing}, vol.~5, no.~4, pp. 653--665, 1976.

\bibitem{doi:10.1137/0210019}
G.~N. Frederickson and J.~Ja'Ja', ``Approximation algorithms for several graph
  augmentation problems,'' \emph{SIAM Journal on Computing}, vol.~10, no.~2,
  pp. 270--283, 1981.

\bibitem{crescenzi1995compendium}
P.~Crescenzi, V.~Kann, and M.~Halld{\'o}rsson, ``A compendium of {NP}
  optimization problems,'' 1995.

\bibitem{chen1989strongly}
W.-T. Chen and N.-F. Huang, ``The strongly connecting problem on multihop
  packet radio networks,'' \emph{IEEE Transactions on Communications}, vol.~37,
  no.~3, pp. 293--295, 1989.

\bibitem{rosenthal1977smallest}
A.~Rosenthal and A.~Goldner, ``Smallest augmentations to biconnect a graph,''
  \emph{SIAM Journal on Computing}, vol.~6, no.~1, pp. 55--66, 1977.

\bibitem{doi:10.1137/0205044}
K.~P. Eswaran and R.~E. Tarjan, ``Augmentation problems,'' \emph{SIAM Journal
  on Computing}, vol.~5, no.~4, pp. 653--665, 1976.

\bibitem{chaudhuri19870}
P.~Chaudhuri, ``An {O}(log n) parallel algorithm for strong connectivity
  augmentation problem,'' \emph{International Journal of Computer Mathematics},
  vol.~22, no. 3-4, pp. 187--197, 1987.

\bibitem{chaudhuri1988fast}
------, ``Fast parallel graph searching with applications,'' \emph{BIT
  Numerical Mathematics}, vol.~28, pp. 1--18, 1988.

\bibitem{aggarwal1989parallel}
A.~Aggarwal, R.~J. Anderson, and M.-Y. Kao, ``Parallel depth-first search in
  general directed graphs,'' in \emph{Proceedings of the twenty-first annual
  ACM symposium on Theory of computing}, 1989, pp. 297--308.

\bibitem{itokawa2007parallel}
T.~Itokawa, A.~Tada, and M.~Migita, ``Parallel algorithm for finding the
  minimum edges to make a disconnected directed acyclic graph strongly
  connected,'' in \emph{Second International Conference on Innovative
  Computing, Information and Control}.\hskip 1em plus 0.5em minus 0.4em\relax
  IEEE, 2007, pp. 131--131.

\bibitem{reed2022scalable}
E.~A. Reed, G.~Ramos, P.~Bogdan, and S.~Pequito, ``A scalable distributed
  dynamical systems approach to learn the strongly connected components and
  diameter of networks,'' \emph{IEEE Transactions on Automatic Control},
  vol.~68, no.~5, pp. 3099--3106, 2022.

\bibitem{ramos2024distributed}
\BIBentryALTinterwordspacing
G.~Ramos and S.~Pequito, ``Distributed computation of source, target, and mixed
  strongly connected components in multi-agent systems,'' in \emph{APCA
  International Conference on Automatic Control and Soft Computing}.\hskip 1em
  plus 0.5em minus 0.4em\relax Springer, 2024, p. To {A}ppear. [Online].
  Available: \url{https://urlis.net/twurbq6g}
\BIBentrySTDinterwordspacing

\bibitem{atman2021distributed}
M.~W.~S. Atman and A.~Gusrialdi, ``Distributed algorithms for verifying and
  ensuring strong connectivity of directed networks,'' in \emph{2021 60th IEEE
  Conference on Decision and Control (CDC)}.\hskip 1em plus 0.5em minus
  0.4em\relax IEEE, 2021, pp. 4798--4803.

\bibitem{atman2022finite}
------, ``Finite-time distributed algorithms for verifying and ensuring strong
  connectivity of directed networks,'' \emph{IEEE Transactions on Network
  Science and Engineering}, vol.~9, no.~6, pp. 4379--4392, 2022.

\bibitem{lozano2008controllability}
R.~Lozano, M.~W. Spong, J.~A. Guerrero, and N.~Chopra, ``Controllability and
  observability of leader-based multi-agent systems,'' in \emph{2008 47th IEEE
  Conference on Decision and Control}.\hskip 1em plus 0.5em minus 0.4em\relax
  IEEE, 2008, pp. 3713--3718.

\bibitem{ge2018survey}
X.~Ge, Q.-L. Han, D.~Ding, X.-M. Zhang, and B.~Ning, ``A survey on recent
  advances in distributed sampled-data cooperative control of multi-agent
  systems,'' \emph{Neurocomputing}, vol. 275, pp. 1684--1701, 2018.

\bibitem{kaminer2007coordinated}
I.~Kaminer, O.~Yakimenko, V.~Dobrokhodov, A.~Pascoal, N.~Hovakimyan, V.~Patel,
  C.~Cao, and A.~Young, ``Coordinated path following for time-critical missions
  of multiple uavs via l1 adaptive output feedback controllers,'' in \emph{AIAA
  Guidance, Navigation and Control Conference and Exhibit}, 2007, p. 6409.

\bibitem{atman2024distributed}
M.~W.~S. Atman and A.~Gusrialdi, ``A distributed algorithm to establish strong
  connectivity in spatially distributed networks via estimation of strongly
  connected components,'' in \emph{2024 European Control Conference
  (ECC)}.\hskip 1em plus 0.5em minus 0.4em\relax IEEE, 2024, pp. 2493--2499.

\bibitem{tarjan1972depth}
R.~Tarjan, ``Depth-first search and linear graph algorithms,'' \emph{SIAM
  journal on computing}, vol.~1, no.~2, pp. 146--160, 1972.

\bibitem{drobyshevskiy2019random}
M.~Drobyshevskiy and D.~Turdakov, ``Random graph modeling: A survey of the
  concepts,'' \emph{ACM computing surveys (CSUR)}, vol.~52, no.~6, pp. 1--36,
  2019.

\end{thebibliography}
}

\end{document}